\documentclass[12pt]{article}
\usepackage{amsmath,amssymb,amsthm}
\usepackage{enumerate}
\usepackage{latexsym}
\usepackage{pictex}
\usepackage[mathscr]{eucal}
\usepackage{graphicx,color,psfrag}

\newtheorem{theorem}{Theorem}[section]
\newtheorem{remark}{Remark}
\newtheorem{example}{Example}
\newtheorem{proposition}[theorem]{Proposition}

\newtheorem{lem}[theorem]{Lemma}

\newtheorem{definition}[theorem]{Definition}

\renewcommand{\epsilon}{\varepsilon}

\begin{document}

\begin{center}
{\Large \bf A remark on the characterization of triangulated graphs}\\
\vskip0.5cm
{\bf R. GARGOURI$^a$\qquad H. NAJAR$^b$}\\
\footnotesize{$^a$ Taibah University, College of Sciences, Department of Mathematics\\
Madinah, Saudi Arabia\\
$^b$ D\'epartement de Math\'ematiques Facult\'e des Sciences de Monastir. Laboratoire de recherche: Alg\'ebre G\'eom\'etrie et  Th\'eorie Spectrale : LR11ES53\\
(Tunisia). Email:hatem.najar@ipeim.rnu.tn }\\
\end{center}
\date{}
\begin{abstract}
In this study we consider the problem of triangulated graphs. Precisely we give a necessary and sufficient condition for a graph to be triangulated. This give an alternative characterization of triangulated graphs.
Our method is based on  the so called perfectly nested sequences.
\end{abstract}
\baselineskip=20pt \setcounter{section}{0}
\renewcommand{\theequation}{\arabic{section}.\arabic{equation}}
\newpage
\section{Introduction}
It is well known that graph theory provides simple, but powerful tools for
constructing models and solving numerous types of interdisciplinary problems and possess a wide range of applications \cite{Bla}. Indeed graphs and graph theory can be used is several areas as software designs, computer networks, social networks, communications networks, information networks,  transportation networks, biological networks, managerial
problems and others.\newline
In 1736, the problem of the Königsberg bridges was considered as the first problem that laid the foundation of graph theory. Since the start of interest in this well known problem several questions and problems have arisen.\newline\vskip0.05cm
Regarding the  topic of this note, the triangulated graphs,  they form an important class among graphs. Since the end of the last sentry a lot of work have been done in the theory of triangulated graphs (which we will define properly below). In some references triangulated graphs are variously called as rigid circuit graphs \cite{Di}, chordal graphs \cite{Ber1} or monotone transitive graphs, like in \cite{Ros1}.\newline \vskip0.01cm
Triangulated  graphs can be characterized in a number of
different ways. See \cite{Bun, Di, Ful, Gav,Hab} and \cite{Ros1}.
We recall that  a vertex $v$ of a graph G is said to be simplicial if $v$ together
with all its adjacent vertices induce a clique in $G$. An ordering
$v_1, v_2,\cdots  v_n $ of all the vertices of $G$ forms a perfect elimination ordering of $G$ if each $v_i, 1 \leq i \leq  n$, is simplicial in the subgraph induced
by $v_i, v_{i+1}, \cdots , v_n$. In \cite{Di}, we find a necessary condition for a graph $G$ to be triangulated which is the existence of simplicial vertex. In \cite{Ful}, Fulkerson and Gross, state that a graph $G$ is triangulated if, and only if, it has a perfect elimination ordering. Precisely, Fulkerson and Gross showed that the class of triangulated
graphs is exactly the class of graphs having perfect elimination orderings.
Thus when the input graph $G$ is not triangulated, no perfect elimination of it
exists.  Rose et al in \cite{Ros} treat the question of triangulated graphs and also give several characterizations of minimal
triangulations.  In \cite{Iba} the author give a new representation of a triangulated graph called the clique-separator graph, whose nodes are the maximal cliques and minimal vertex separators of the graph.\newline At the end of this section we mention that triangulated graphs have applications in sevral areas such as computer
vision \cite{Chu}, the solution of
sparse symmetric systems of linear equations \cite{Ros2}, data-base management
systems \cite{Tar} and  knowledge based systems \cite{Eng}. At the end of the paper we collect two main consequences of triangulated graphs. Another consquence for triangulated graphs is the problem of finding a maximum clique. Indeed,  in a triangulated graph we get the answer in polynomial-time, while the same problem for general graphs is NP-complete. More generally, a triangulated  graph can have only linearly many maximal cliques, while non-chordal graphs may have exponentially many \cite{Kat}.
\subsection{Basic concept of graph theory}
In this section we  will enumerate and explain the basic
definitions and the necessary terminology to make use
of graph theory. There is a great variety in how different authors presented the basic definitions of the graph theory. Indeed there are many roughly equivalent definitions of a graph.\newline  Most commonly, a graph ${\displaystyle G}$ is defined as an ordered pair ${\displaystyle (V,E)}$ , where ${\displaystyle V=\{v_{1},\cdots ,v_{n},\cdots \}}$ is called the graph's vertex-set or some times the node set and  ${\displaystyle E=\{e_{1},\ldots ,e_{m},\cdots \}\subset \{\{x,y\}|x,y\in V\}} $ is called the edges set. \newline Given a graph ${\displaystyle G} ,$ we often denote the vertex—set by ${\displaystyle V(G)} $ and the edge—set by ${\displaystyle E(G)}$. To visualize a graph as described above, we draw  dots corresponding to vertices $ {\displaystyle v_{1},\ldots ,v_{n},\cdots }$. Then, for all $ {\displaystyle i,j\in \{1,\cdots ,n,\cdots \}}$ one imagine  a line between the dots corresponding to vertices ${\displaystyle v_{i},v_{j}} $ if and only if there exists an edge ${\displaystyle \{v_{i},v_{j}\}\in E}$. Note that the placement of the dots is generally unimportant; many different pictures can represent the same graph as it is given in the example below:

\begin{example}
\end{example}
\begin{center}
  \includegraphics[scale=1]{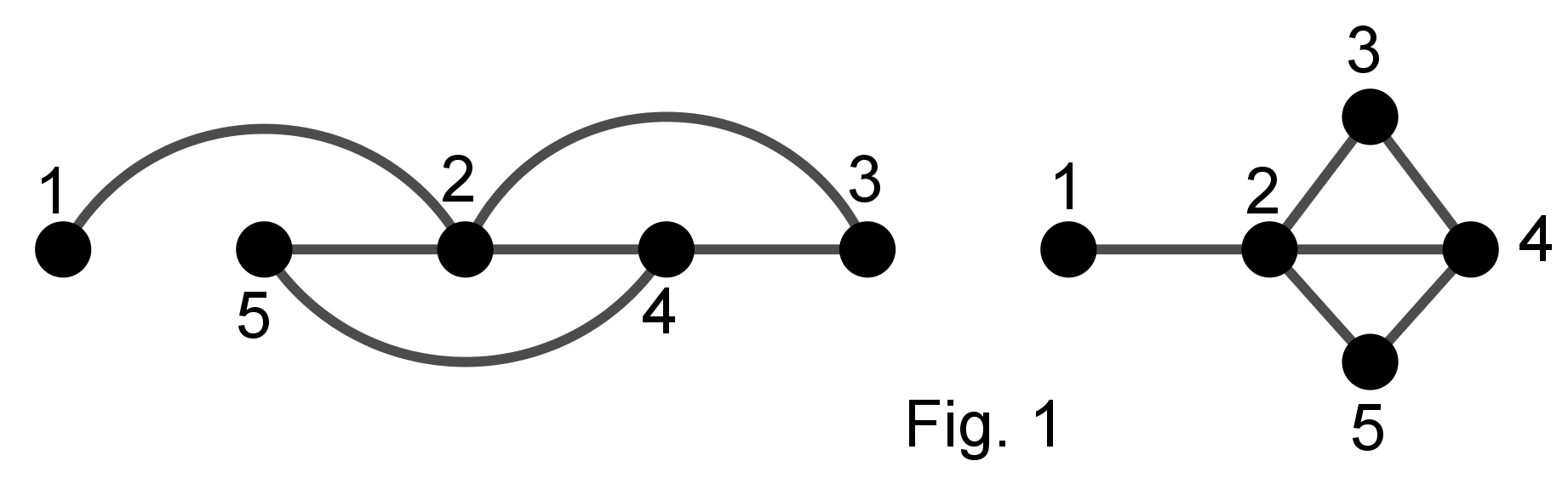}
\end{center}
A subgraph is a concept akin to the subset. A subgraph has a subset of the vertex set $A\subseteq V,$ a subset $E(A)=\{ \{x,y\}\in E: x,y\in A\}$ of the edge set $ E,$ and each edge's endpoints in the larger graph has the same edges in the subgraph. We denote it by $G(A)=(A,E(A))$. \newline
Two vertices are said to be adjacent if there is an edge joining them. Given $x\in V,$ the set of all adjacent vertices in $G$ is denoted by $Adj(x),$
\begin{equation}
Adj(x)=\{y\in V, \{x,y\}\in E\}.
\end{equation} The word incident has two meanings:
On the one hand, an edge $e$ is said to be incident to a vertex $v$ if $v$ is an endpoint of $e$.
On the other hand, two edges are also incident to each other if both are incident to the same vertex. A set $C$ of vertices is a clique if every pair of vertices in $C$ are adjacent. A clique of a graph $G$ is a complete subgraph of $G$.\\
A path is a sequence of edges ${\displaystyle <e_{1},...,e_{N}>} $  (also denoted $(v_1,\cdots ,v_n)$ such that $e_i$ is adjacent to $e_{i+1}$ for all $i$ from $1$ to $N-1, e_i$ relates $v_i$ to $v_{i+1}$.  Two vertices are said to be connected if there is a path connecting them. A cycle is a path such that the last edge of the path is adjacent to the first and visit each vertices once (in some references they call this elementary cycle).\newline
In a simple graph each edge connects two different vertices and no
two edges connect the same pair of vertices.  Multi-graphs may have multiple edges connecting the same two vertices. An edge that connects a vertex to itself is called a loop. Two graphs $ G$ and $G'$ are said to be isomorphic if there is a one-to-one function from (or, if you prefer, one-to-one correspondence between) the vertex set of $ G$ to the vertex set of $G'$ such that two vertices in $G$ are adjacent if and only if their images in $G'$ are adjacent. Technically, the multiplicity of the edges must also be preserved, but our definition suffices for simple graphs, which are graphs without multiple edges or loops.
\subsubsection{Definitions}
\begin{definition}
A graph is called {\bf{triangulated}} if every cycle of length greater than three possesses a chord, i.e. an edge joining two non-consecutive vertices of the cycle.
\end{definition}
\begin{definition}
A vertex $x$ of a graph $G=(V,E)$ is called perfect in $G$ if $Adj(x)=\emptyset$ or $(\{x\}\cup Adj(x))$ is a clique. For $A\subseteq V,$ we denote by $P(A)$ the set of all perfect vertices of $A$ in $G(A)$.
\end{definition}
\begin{example}
\end{example}
\begin{center}
  \includegraphics[scale=0.9]{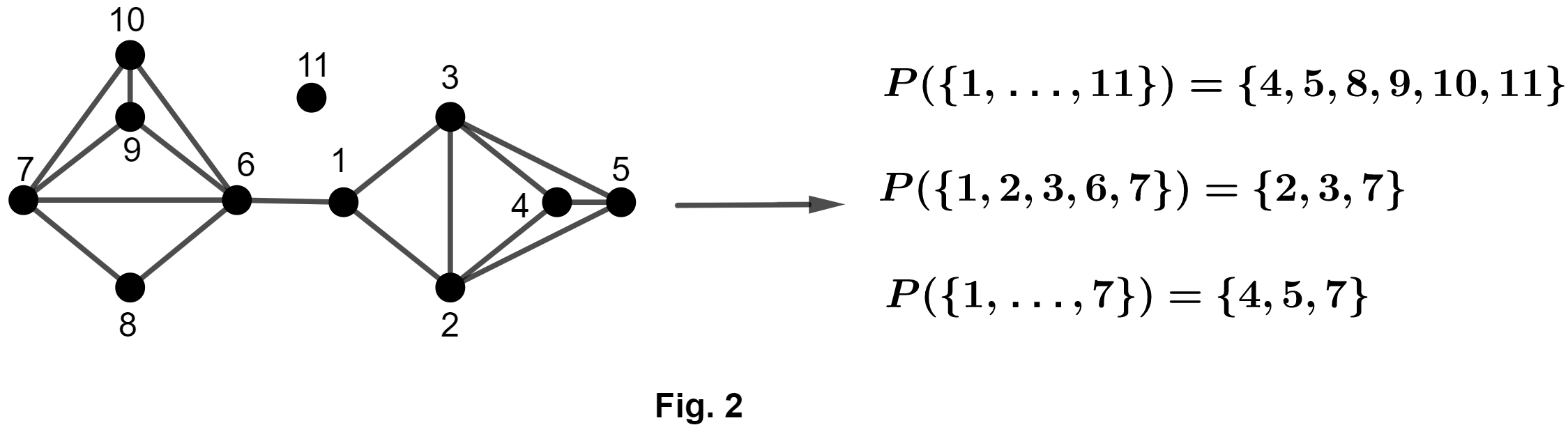}
\end{center}

\begin{definition} Let $G=(V,E),$be a graph. A sequence $(U_n)_{n\in \mathbb{N}}$ of subsets of $\mathcal{P}(V),$ is said to be {\bf{perfectly nested }} on on $G=(V,E),$ if it satisfies the following three conditions
\begin{enumerate}
\item $U_0=V$.
\item $\forall n\in \mathbb{N},$ $$U_{n+1}\subseteq U_n.$$
\item $\forall n\in \mathbb{N}, P(U_n)\neq \emptyset$, and
$$U_n\setminus U_{n+1}\subseteq P(U_n).$$
\end{enumerate}
We say that the sequence is {\bf{stationary  perfectly nested}}, if furthermore the three last conditions, there exists $n_0\geq 0$ such that we have $P(U_n)=U_n,$ for any $n\geq n_0$.
\end{definition}
\section{The results}
\begin{proposition}\label{prop1}
Let $G=(V,E)$ be a graph and $A\subset P(X).$ Then the following items are equivalents:
\begin{enumerate}
\item $G=(V,E)$ is triangulated;
\item $G=(V\backslash A)$ is triangulated.
\end{enumerate}
\end{proposition}
\begin{proof}
It is clear that $1)$ implies $2)$.\newline
For the other sense, let us consider $A\neq \emptyset$. Let $C=(v_1,\cdots ,v_n)$ be a cycle in $\displaystyle G=(V,E)$.
There are two possibilities
\begin{enumerate}[(a)]
\item If $C\subseteq V\setminus A,$ then $C$ has a chord.
\item If $C \cap A\neq \emptyset,$ there exists $i_0\in\{1,\cdots, n\}$ such that $v_{i_0}\in A$. As $v_{i_0}$ is a perfect vertex, then $\displaystyle Adj(v_{i_0})\cap C$ is a clique and so, $C$ has a chord. So $G=(V,E)$ is triangulated.
\end{enumerate}
\end{proof}

\begin{theorem}\label{tho1}
  Let $G=(V,E)$ be a graph. Suppose that there exists  a perfectly nested sequence on $G=(V,E)$. Then $\displaystyle G=(V,E)$ is a triangulated graph.
\end{theorem}
\begin{proof}
Let $C=(v_1,v_2,\cdots, v_n), n\geq 4,$ be  a cycle in the graph $G=(V,E)$. Let $(U_n)_n$ a perfectly nested sequences.  There exists $n\in \mathbb{N},$ such that
$$C\subseteq U_n,C\nsubseteq U_{n+1}.$$
So,
$$C \cap P(U_n)\neq \emptyset.$$
This ensures that there exists a perfect vertices $x\in C\cap G(U_n)$. So $C$ has a chord.
\end{proof}

\begin{remark} When $V$ is infinite let us notice that we can construct triangulated graphs which do not have a perfectly nested sequence. Indeed for $V=\mathbb{Z}$ and $E=\{\{n,n+1\}, n\in \mathbb{Z}\}, G=(V,E)$ is triangulated and $P(V)=\emptyset$.
\end{remark}

\begin{theorem}\label{tho2}
Let $G=(V,E)$ be a graph, with $V$ being a finite set. Then, $G=(V,E)$ is triangulated if and only if there exists a stationary perfectly nested sequence
on $G=(V,E)$.
\end{theorem}
\begin{proof}
For the proof, we need the following two basic lemmas:

\begin{lem}\label{per} \cite{Ros1}
Let $G=(V,E)$ be a triangulated finite graph. Then $P(V)\neq \emptyset$.
\end{lem}
\begin{remark}
  The proof of the last lemma is given in \cite{Ros1} by using the notion of the minimal separators and elimination  process. From a different point of view we can see this by noting that if we suppose that $P(V)=\emptyset,$ starting with a non perfect point $x$ it has necessary two adjacent points $x_1,y_1$ which them selves are non adjacent to each others. As $V$ is a finite set, a typical end of the process should be in the form of Fig1. At this step, as no point can be adjacent to a single point being a non perfect point it is adjacent to more than two points. This leads forcibly the existence of a non-chordal cycle of length more than $3$. So $G=(V,E)$ is not triangulated.
\end{remark}
\begin{lem}\label{con}
Let $G=(V,E)$ be a connected graph. Let $A\subset P(V)$. Then,
$\displaystyle G=(V\setminus A),$ is also a connected graph.
\end{lem}
\begin{proof}
Let $A\subseteq P(V),$ and  $v_1,v_n\in V\setminus A$. As $G=(V,E)$ is a connected graph, there exists a path $P=(v_1,\cdots, v_n)$ in $\displaystyle G=(V,E),$ for any $v_{i_0}\in A\cap P,$ as $v_{i_0-1}$ and $v_{i_0+1}$ are in $Adj(v_{i_0}),$ we get that $v_{i_0-1}\in Adj(V_{i_0+1},$ by the definition of $v_{i_0}$ being a perfect point. So we get a path $P'=(v_1,\cdots,v_{i_0-1},v_{i_0+1},\cdots, v_n)$ in $\displaystyle G=(V\setminus A,E(A))$.
\end{proof}
Let us start by the proof of the necessary condition. By Lemma \ref{per} we know that when $G=(V,E)$ is triangulated, then $P(V)\neq \emptyset$; For $V_1\varsubsetneq P(V)$, we set
$$U_2=U_1\setminus V_1, {\rm{with}}\ U_1=V.$$
By Proposition \ref{prop1}, $G(U_2,E(U_2))$ is a triangulated graph, so $P(U_2)\neq \emptyset$. Let $V_2\subseteq P(U_2),$ we set
$$U_3=U_2\setminus V_2. $$
 In the same way we construct the perfectly nested sequence. As the set $V$ is finite by assumption we get the stationarity property using the result of Lemma \ref{con}, as at the end
it remains only one or two perfect points.\newline
For the sufficient condition it is given by Theorem \ref{tho1}
\end{proof}
Below we give an example where using our result we get the answer to the question of triangulated graph after only three steps.
\begin{example}\label{ex3}

\begin{center}
  \includegraphics[scale=1]{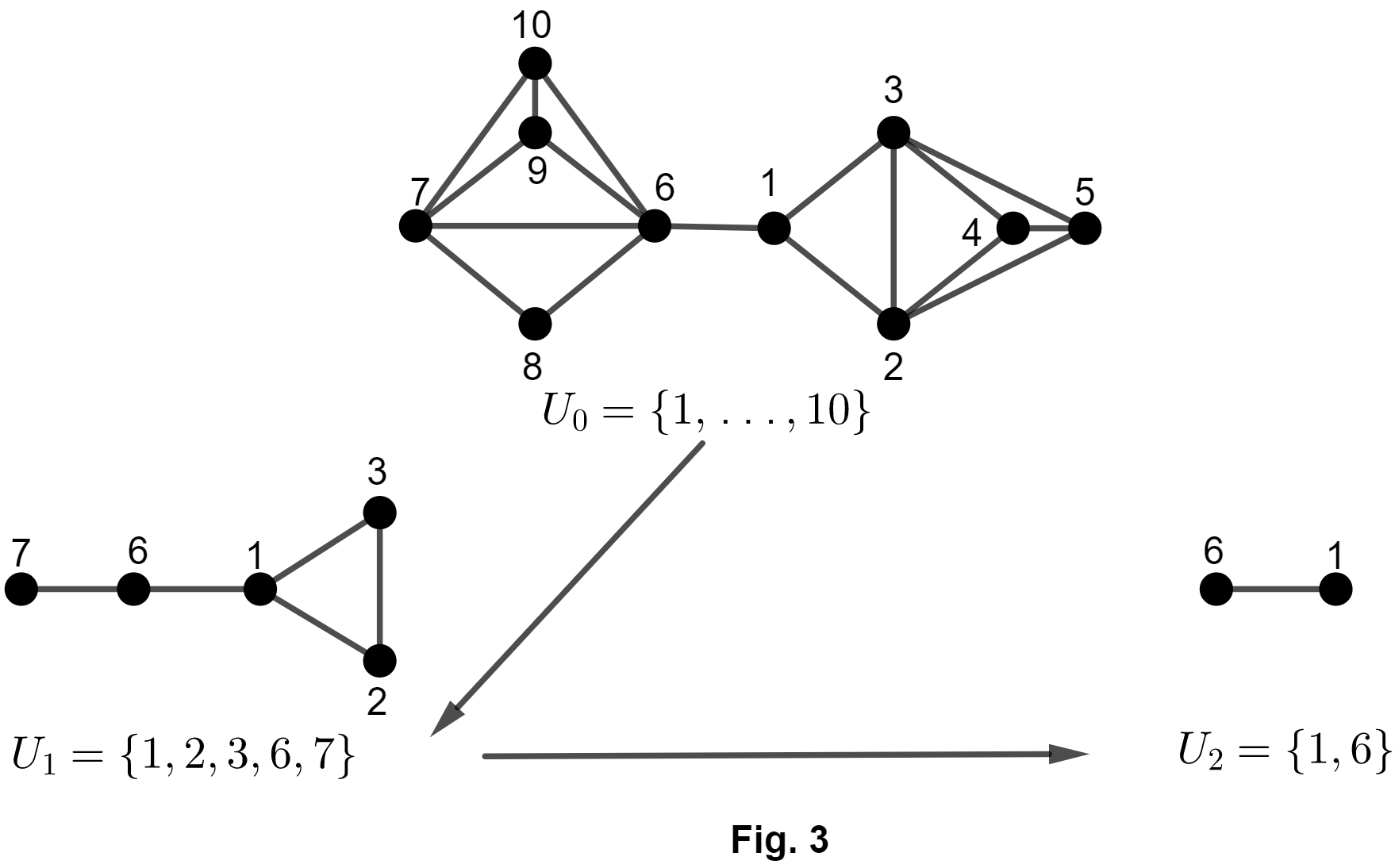}
\end{center}

\end{example}

Let us end this section by the following remark.
\begin{remark}
In the particular case, if we take a perfectly nested sequence in the Theorem \ref{tho2}  with the property that for any $1\leq i <n$ , $\displaystyle |U_i\setminus U_{i+1}|=1,|U_{n}|=1$ we get the characterization given in \cite{Ros1}, by taking
$$\alpha: \{1,\cdots ,n\}\mapsto V.$$
$$\alpha^{-1}(U_i\setminus U_{i+1})=\{i\},\alpha^{-1}(U_{n})=\{n\} .$$
\end{remark}
\section{Some consequences of triangulation}
For completeness, below we collect some possible implications of our main result. In addition to the consequence given in the introduction which concerns  the answer in polynomial-time to some problems for triangulated graphs, while the same problem for general graphs is NP-complete. Below we collect two more consequences of triangulated graphs.
\subsection{Directed Acyclic Graphs}
An orientation $D$ of a finite graph $G$ with $n$ vertices is obtained by
considering a fixed direction, either $x \to y$ or $y \to  x$, on every edge $\{xy\}$
of $G$.\\
We call an orientation $D$  acyclic if there does not exist any directed
cycle. A directed graph having no directed cycle is
known as a directed acyclic graph, we write DAG for short. DAGs provide
frequently used data structures in computer science for encoding dependencies.
The topological ordering is another way to describe a DAG.
A topological ordering of a directed graph $G=(V,E)$ is an ordering of
its vertices as $v_1, v_2, \cdots , v_n$
such that for every arc $v_i \to  v_j$ , we have
$i < j$.\newline
Let us consider an acyclic orientation $D$ of $G$. An edge of $D$, or
 is said to be dependent (in D) if its reversal creates
a directed cycle in the resulted orientation. Note that $v_i \to v_j$ is a
dependent arc if and only if there exists a directed walk of length
at least two from $v_i$ to $v_j$. We denote by $d(D)$, the number of dependent
arcs in $D$.
\begin{definition}
A graph $G$ is called fully orientable if it has an acyclic orientation
with exactly $d$ dependent arcs for every number $d$ between  $dmin(G)$ and $dmax(G)$, the minimum
and the maximum values of $d(D)$ over all acyclic orientations $D$ of
$G$.
\end{definition}
\begin{example}
An acyclic orientation with $6$ dependents arcs.
\end{example}

\begin{center}
  \includegraphics[scale=0.5]{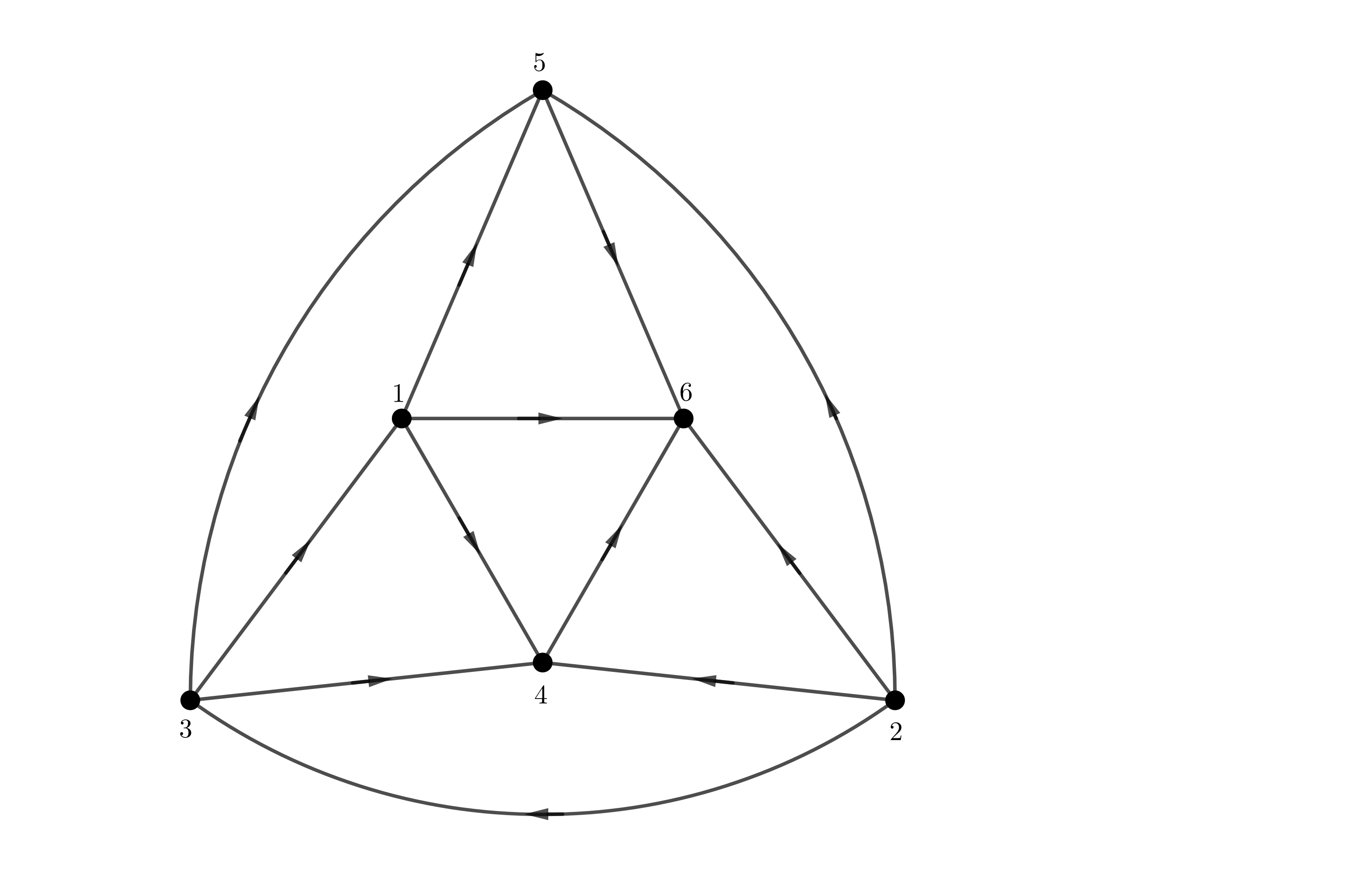}
\end{center}

In \cite{Lai}, the authors show that all chordal
graphs are fully orientable. Let us denote the complete $r$-partite graph each of whose partite
sets has $n$ vertices by $K_r(n)$. As it is also known \cite{Cha} that $K_r(n)$
is not fully
orientable when $r \geq  3$ and $n \geq 2$. One deduces that the acyclic orientation $K_3(2)$ given in the example is not a triangulated graph.
If we consider the graph of example \ref{ex3}, as it is a triangulated graph we deduce that it is a fully orientable graph.

\subsection{Chromatic Number}
 A graph coloring is an assignment of labels, called colors, to the vertices of a graph such that no two adjacent vertices share the same color.
\begin{definition}
\begin{enumerate}
\item A Clique number $\omega(G),$ is the maximum size of a clique in $G$.
\item A Chromatic number $\chi(G),$ is the minimum coloring number.
\end{enumerate}

\end{definition}
\begin{remark} For a general graph $G,$ we have
  \begin{equation}\label{R1} \chi(G)\geq \omega(G).
\end{equation}For triangulated graphs we have equality in equation $(\ref{R1})$.
\end{remark}


\begin{thebibliography}{Tototo}

\bibitem{Kat} K. Asdre and  S. D. Nikolopoulos, {\sl\textit{ NP-completeness results for some problems on subclasses of
bipartite and chordal graphs.}} Theo. Compu. Science {\bf{(381)}} (2007)  p. 248-259.
\bibitem{Ber1} C. Berge, {\sl\textit{Farbung von Graphen, deren samtliche bzw. deren ungerade Kreise starrsind}}, Wiss. Z. Martin-Luther-Univ. Halle-Wittenberg Math.-Natur. Reihe, (1961) p 114.
    \bibitem{Ber2} C. Berge, {\sl\textit{some classes of perfect graphs, in Graph Theory and theorical Physics}}, (F. Harry, Ed), p. 155-165, Academic Press, Newyork, (1967).
\bibitem{Bla} J. R. S. Blair and B. W. Peyton, {\sl\textit{ An introduction to chordal graphs
and clique trees.}}  Graph Theory and Sparse Matrix Computations,
J. A. George, J. R. Gilbert, and J. W. H. Liu, eds., Springer Verlag, p. 1-30. IMA Volumes in Mathematics and its Applications,
Vol. 56. (1993).
\bibitem{Bun} P. Buneman, {\sl\textit{A characterization of rigid circuit graphs.}} Discrete
Math. 9 (1974)  p. 205-212.
\bibitem{Cha} G. J. Chang, C.-Y. Lin, and L.-D. Tong, {\sl\textit{Independent arcs of
acyclic orientations of complete r-partite graphs,}} Discrete Math.
309 (2009) p. 4280-4286.
\bibitem{Chu} F. R. K. Chung and D. Mumford,{\sl\textit{ Chordal completions of planar
graphs.}} J. Combin. Theory Ser. B 31 (1994) p. 96-106.
\bibitem{Di} G. A. Dirac,  {\sl\textit{On rigid circuit graphs}} , Abhandlungen aus dem Mathematischen Seminar der Universität Hamburg, 25: (1961) p 71-76.
\bibitem{Eng} R. E. England, J. R. S. Blair, and M. G. Thomason, {\sl\textit{Independent
computations in a probabilistic knowledge-based system}}
Technical Report CS-91-128, Department of Computer Science,
The University of Tennessee, Knoxville, Tennessee, (1991).
\bibitem{Ful} D.R. Fulkerson and O.A. Gross. {\sl\textit{Incidence matrices and interval graphs.}} Pacific
J. Math., 15  (1965) p. 5335-855.
\bibitem{Gav} F. Gavril, {\sl\textit{ The intersection graphs of subtrees in trees are exactly
the chordal graphs}} J. Combin. Theory Ser. B 16  (1974) p. 47-56.
\bibitem{Hab} M. Habib and V.Limouzy
{\sl\textit{On some simplicial elimination schemes for chordal graphs. Electronic Notes in Discrete Mathematics}}  32: p. 125-132 (2009)
\bibitem{Iba} L. Ibarra {\sl\textit{The clique-separator graph for chordal graphs.}}
Disc. App. Math. 157 (2009) p. 1737-1749.
\bibitem{Lai} Hsin-Hao Lai and  Ko-Wei Lih,
{\sl\textit{Chordal Graphs are Fully Orientable.}} Ars Comb. 122: p. 289-298 (2015)
\bibitem{Ros1} D. J. Rose, {\sl\textit{ Triangulated Graphs and the Elimination Process.}} Jour. Math. Anal. App. {\bf{(3)}}(1970) p. 597-609.
\bibitem{Ros2} D. J. Rose, {\sl\textit{ A graph-theoretic study of the numerical solution
of sparse positive definite systems of linear equations, in Graph
Theory and Computing.}} (R. C. Read, ed.), p. 183-217, Academic
Press, New York, (1972).
\bibitem{Ros} D. J. Rose, R. E. Tarjan, and G. S. Lueker, {\sl\textit{ Algorithmic aspects
of vertex elimination on graphs.}} SIAM J. Comput., 5 (1976)  p. 266-283.
\bibitem{Tar} R. E. Tarjan and M. Yannakakis. {\sl\textit{Simple linear-time algorithms
to test chordality of graphs, test acyclicity of hypergraphs, and
selectively reduce acyclic hypergraphs.}} SIAM J. Comput. 13
(1984) p. 566-579.
\end{thebibliography}
\end{document}